\documentclass[reqno]{amsart}
\usepackage{amsmath, amssymb, amsthm, geometry, enumerate, graphicx, amsfonts, hyperref, mathrsfs}
\hypersetup{colorlinks=true,linkcolor=red, anchorcolor=blue, citecolor=blue, urlcolor=red, filecolor=magenta, pdftoolbar=true}
\theoremstyle{plain}
\newtheorem{thm}{\it Theorem}[section]
\newtheorem{prop}[thm]{\it Proposition}
\newtheorem{cor}[thm]{\it Corollary}
\newtheorem{lem}[thm]{\it Lemma}
\theoremstyle{remark}
\newtheorem{defn}[thm]{Def{}inition}
\newtheorem{rem}[thm]{Remark}
\newtheorem{exa}[thm]{Example}
\numberwithin{equation}{section}
\usepackage{enumerate}
\usepackage{xcolor}
\allowdisplaybreaks
\begin{document}
\title [Matrix-valued Riesz bases]{Operators for matrix-valued Riesz bases over LCA groups}

\author[Jyoti]{Jyoti}
	\address{{\bf{Jyoti}}, Department of Mathematics,
		University of Delhi, Delhi-110007, India.}
	\email{jyoti.sheoran3@gmail.com}

\author[Lalit   Kumar Vashisht]{Lalit  Kumar  Vashisht$^*$}
	\address{{\bf{Lalit  Kumar  Vashisht}}, Department of Mathematics,
		University of Delhi, Delhi-110007, India.}
	\email{lalitkvashisht@gmail.com}

\begin{abstract}
The image of a given orthonormal basis for a separable Hilbert space $\mathcal{H}$ under a bijective, bounded, and linear operator acting on $\mathcal{H}$ is called a Riesz basis of $\mathcal{H}$.
 Contrary to what happens with Riesz bases (in the usual sense) in separable Hilbert spaces, it is not true in general that the image of a matrix-valued orthonormal basis under a bounded, linear, and bijective operator on  $L^2(G, \mathbb{C}^{s\times r})$ is also a basis and frame for the space $L^2(G, \mathbb{C}^{s\times r})$, where $G$ is a $\sigma$-compact and  metrizable locally compact abelian (LCA) group. We give   some classes of operators for the construction of matrix-valued Riesz bases from orthonormal bases of the space $L^2(G, \mathbb{C}^{s\times r})$. Motivated by a result due to  Holub, we show that a bounded, linear, and  bijective operator acting on $L^2(G, \mathbb{C}^{s\times r})$ which is adjointable with respect to the matrix-valued inner product  is positive if and only if it maps a matrix-valued Riesz basis of the space $L^2(G, \mathbb{C}^{s\times r})$ to its dual Riesz basis.

\end{abstract}


\keywords{ Riesz basis, orthonormal basis, frames, positive operator, locally compact group.\\
The research of Jyoti is supported by the  WISE-PDF research grant of WISE--KIRAN Division, Department of Science and Technology (DST), Government of India (Grant No.: DST/WISE-PDF/PM-6/2023 ).  Lalit \break  Kumar Vashisht is  supported by the Faculty Research Programme Grant-IoE, University of Delhi \ (Grant No.: Ref. No./IoE/2023-24/12/FRP).\\
$^*$Corresponding author}

\maketitle

\baselineskip15pt
\section{Introduction}
 Riesz bases are ubiquitous in  basis theory for separable Hilbert spaces, leading further to new developments in important areas of  analysis such as stable synthesis and analysis of functions (details below) in the underlying Hilbert space.
A collection of vectors $\{x_k\}_{ k \in \mathbb{N}}$ in a separable  Hilbert space  $\mathcal{H}$ is  called a \emph{Riesz basis} for $\mathcal{H}$ if $x_k = Te_k$ for all $k \in \mathbb{N}$, where $T$ is a bounded, linear, and bijective operator acting on $\mathcal{H}$ and $\{e_k\}_{ k \in \mathbb{N}}$ is an orthonormal basis for $\mathcal{H}$.   Heil \cite{Heil}, Novikov, Protasov and Skopina \cite{NPS}, Young \cite{Young} are good texts for basic theory of Riesz bases. Recent work on different types of Riesz bases can be found in Favier and  Zalik \cite{FZ} and  Jyoti and Vashisht \cite{JV23}.  A Riesz basis for $\mathcal{H}$ is also a frame that gives a series representation, not necessarily unique, of each vector in $\mathcal{H}$. To be precise, if $\{x_k\}_{ k \in \mathbb{N}}$ is a Riesz basis for $\mathcal{H}$, then the map $S: \mathcal{H}  \rightarrow \mathcal{H}$ given by $Sx= \sum_{k \in \mathbb{N}}\langle x, x_k\rangle x_k$ is  bounded, linear, self-adjoint, positive and invertible operator on $\mathcal{H}$. This gives, $x = S S^{-1}x = \sum_{k \in \mathbb{N}}\langle x, S^{-1}x_k\rangle x_k$ for all $x \in \mathcal{H}$. Thus, a frame for $\mathcal{H}$ is complete, that is $\overline{\text{span}}\{x_k\}_{ k \in \mathbb{N}}= \mathcal{H}$. We refer to texts by  Heil \cite{Heil},  Young \cite{Young}, and a research article by  Heil and Walnut \cite{Heil2} for the basics of frames and their connection to Riesz bases.

 Holub, in \cite{Holub}, studied Riesz bases in separable Hilbert spaces. He characterized Riesz bases in terms of operators associated with frames for separable Hilbert spaces. Holub  proved in \cite{HolubII} that a normalized Riesz basis for a separable Hilbert space $H$ induces a new, but equivalent inner product, in which it is an orthonormal basis.  A corollary of the technique of proof of this result offers a series representation for all positive isomorphisms on a Hilbert space. On the other hand, Antol\'{i}n and Zalik \cite{AZ}, Xia \cite{XIA2} showed the importance of matrix-valued wavelets and orthonormal bases in matrix-valued function spaces in the study of matrix-valued signals. Recent development on matrix-valued orthonormal bases and frames over the euclidean space $\mathbb{R}^d$ can be found in \cite{JV1}. The rise of frame theory has led to new applications of Riesz bases, for example, in signal processing. To be exact, a signal (function) can be decomposed and analyzed in terms of Riesz bases.  One of the most important recent advances concerns finding classes of operators that can preserve or generate bases and frames in signal spaces by their action on a given vector, basis or frame in the space \cite{AA1, AA3}. Given  a  $\sigma$-compact and  metrizable locally compact abelian (LCA) group $G$, matrix-valued Riesz bases and frames of the space $L^2(G, \mathbb{C}^{s\times r})$ differ greatly from the standard Riesz bases and frames, respectively, in separable Hilbert spaces. This is illustrated by examples and applications in the recent study on  matrix-valued frames over the euclidean space $\mathbb{R}^d$ in \cite{JVK00, JV1}.

 Motivated by the  works mentioned above, we study matrix-valued Riesz basis over LCA groups. The first aim of this paper is to provide classes of operators that give matrix-valued Riesz bases from matrix-valued orthonormal bases for  $L^2(G, \mathbb{C}^{s\times r})$. Other notable contribution is  a    characterization of  positive operators acting on $L^2(G, \mathbb{C}^{s\times r})$ that are adjointable with respect to the matrix-valued inner product. This is inspired by a related result for standard Riesz bases in separable Hilbert spaces due to Holub \cite{HolubII}.  This study  is important because the frame operator of a frame is also positive, and the study of operators that give frame conditions is a hot topic for research in mathematical physics and operator theory \cite{AA3, Heil}.

 The structure of the paper is  as follows: To make the paper self-contained, Section \ref{secII} gives some basic definitions and results from operator theory and matrix-valued function spaces over LCA groups.
  In Section \ref{secIII}, we give  main results. First, we give some classes of operators for the  construction of matrix-valued Riesz bases from   matrix-valued orthonormal bases  of the space  $L^2(G, \mathbb{C}^{s\times r})$, see
  Theorem \ref{thpi},  Theorem \ref{thm3.3I}. The adjointability of the square root of positive operators acting on $L^2(G, \mathbb{C}^{s\times r})$, with respect to the matrix-valued inner product, can be found in Proposition \ref{prosq}.  Interplay between matrix-valued Riesz basis in terms of positive operators  on $L^2(G, \mathbb{C}^{s\times r})$ is given in Theorem \ref{thpsq}. The construction of two matrix-valued Riesz bases from a given matrix-valued Riesz basis of $L^2(G, \mathbb{C}^{s\times r})$ is given in Theorem \ref{thmpol}. As an application of above results Theorem \ref{thmpwr}, Theorem \ref{thmmisc24}--Theorem \ref{clarefii} give operators on $L^2(G, \mathbb{C}^{s\times r})$ for the construction of matrix-valued Riesz basis. Finally, Theorem \ref{mainthmI} provides a characterization of positive operators on  $L^2(G, \mathbb{C}^{s\times r})$  that are adjointable with respect to the matrix-valued inner product in terms of matrix-valued Riesz bases.
\section{Preliminaries}\label{secII}
This section gives basic definitions and collects some results about positive operators and matrix-valued function spaces over LCA groups.
\subsection{Positive operators}
  As is standard, the space of bounded linear operators acting on $\mathcal{H}$ is denoted by $\mathcal{B}(\mathcal{H})$. The  \emph{Hilbert-adjoint operator} of an operator $T \in \mathcal{B}(\mathcal{H})$ is an operator $T^{*} \in \mathcal{B}(\mathcal{H})$, defined uniquely by the equation:
$\langle T x, y\rangle = \langle x, T^{*}y\rangle \ \text{for all} \ x$, $y \in \mathcal{H}$.
In this case, we say that $T$ is \emph{adjointable} with respect to the inner product $\langle \cdot, \cdot \rangle$.
A self-adjoint operator $T \in \mathcal{B}(\mathcal{H})$ is said to be  \emph{positive}  if $\langle Tx,x \rangle \ge 0 $ for every $x \in \mathcal{H}$. If there exists an operator $W$ such that $W^2=T$, then $W$ is known as a \emph{square root} of $T$. The following result guarantees the existence of at least one square root of every positive operator. This can  be found in any standard text on analysis, e.g., Heuser \cite{HH}.
\begin{thm}\cite{HH}\label{thpo}
Let $T \in \mathcal{B}(\mathcal{H})$ be positive. Then, there exists a unique bounded and positive operator $W$ such that $W^2=T$. The operator $W$ is known as the square root of $T$ and is denoted by $T^{1/2}$. The operator $T^{1/2}$ can be expressed as a limit of a sequence of polynomials in T. Moreover, $T^{1/2}$  is invertible if $T$ is invertible.
\end{thm}
\begin{rem}\label{remsq}
Since $\|T\|^2=\|TT^*\|$ for $T \in \mathcal{B}(\mathcal{H})$, we have
$\|T^{1/2}\|^2=\|T^{1/2}(T^{1/2})^*\|=\|T^{1/2}T^{1/2}\|=\|T\|$.
Hence, $\|T^{1/2}\|=\|T\|^{1/2}$.
\end{rem}

We end the section with following lemma which gives some algebraic properties  of positive operators.
\begin{lem}\cite{HH}\label{lem1a}
Let $S$, $T \in \mathcal{B}(\mathcal{H})$ be positive and $I$ is the identity operator on $\mathcal{H}$. Then, $I+S$ is an invertible operator on $\mathcal{H}$, $T+S$ is positive. Further, if $TS=ST$, then $TS$ is  positive.
\end{lem}
\subsection{The space $ L^2(G, \mathbb{C}^{s\times r})$}
Throughout the paper,  $\mathbb{Z}$, $\mathbb{N}$, $\mathbb{R}$  and $\mathbb{C}$ denote the set of integers, positive integers, real numbers and complex numbers, respectively. Symbol $G$ denotes a locally compact abelian (LCA) group which is metrizable and $\sigma$-compact. Let $\mu_{G}$ be the Haar measure on $G$.  Folland \cite{Foll}, Hewitt and  Ross \cite{Hewit} are good references for fundamental properties of locally compact groups.   Matrix-valued functions are denoted by bold letters. For $ s \in \mathbb{N}$, $\textbf{O}_{s\times s}$ and $\textbf{I}_{s\times s}$ denote the zero matrix and the identity matrix of order $s$-by-$s$,  respectively.  $\textbf{O}$ denotes the zero matrix. $M_s(\mathbb{C})$ denotes the space of all complex matrices of order $s$-by-$s$. A diagonal matrix with entries $a_i$ is denoted by diag$(a_1, a_2,\cdots, a_k)$.  As is well known, the class of   Borel functions $f$ on $G$ such that  $\int_{G} |f|^2 d\mu_{G} < \infty$ is denoted by $L^2(G)$. The space  $L^2(G)$ is a  Hilbert space with respect to the standard inner product $\langle f, g\rangle = \int_{G} f \overline{g} d\mu_{G}$.

We define
\begin{align*}
 L^2(G, \mathbb{C}^{s\times r}) := \Big\{\mathbf{f} = \big[f_{i j}\big]_{1 \leq i \leq s \atop 1 \leq j \leq r}
 :  f_{ij} \in  L^2(G)\ (1 \leq i \leq s,1 \leq j \leq r)\Big\},
\end{align*}
where $\big[f_{i j}\big]_{1 \leq i \leq s \atop 1 \leq j \leq r}$ is a matrix of order $s$-by-$r$ with entries $f_{i j}$.  The  Frobenius norm on $L^2(G, \mathbb{C}^{s\times r})$ is given by
 \begin{align}\label{eqnorm1}
\|\mathbf{f}\| = \Big(\sum\limits_{1 \leq i \leq s \atop 1 \leq j \leq r} \int_{G}|f_{ij}|^2 d\mu_{G} \Big)^{\frac{1}{2}}.
\end{align}

The integral of   a function
 $\mathbf{f} = \big[f_{i j}\big]_{1 \leq i \leq s \atop 1 \leq j \leq r} \in L^2(G, \mathbb{C}^{s\times r})$  is defined as $\int_{G}\mathbf{f}d\mu_{G} = \left[\int_{G}f_{i j}d\mu_{G}\right]_{1 \leq i \leq s \atop 1 \leq j \leq r}$.

For $\mathbf{f}$,  $\mathbf{g} \in L^2(G, \mathbb{C}^{s\times r})$, the matrix-valued inner product is defined as
 \begin{align}\label{dminp}
\langle \mathbf{f},  \mathbf{g}\rangle = \int_{G}\mathbf{f}(x) \mathbf{g}^*(x) d\mu_{G}.
\end{align}
Here, $\mathbf{g}^*$ denotes the conjugate transpose of $\mathbf{g}$. One may observe that the matrix-valued inner product defined  in  \eqref{dminp}  is not an  inner product in usual sense, but still satisfies the following properties:
\begin{enumerate}[$(i)$]
  \item $\langle \mathbf{A} \mathbf{f} + \mathbf{B} \mathbf{g},  \mathbf{h}\rangle = \mathbf{A} \langle \mathbf{f},  \mathbf{h}\rangle + \mathbf{B}  \langle \mathbf{g},  \mathbf{h}\rangle $ for all  $\mathbf{f}$,  $\mathbf{g}$,
   $\mathbf{h} \in L^2(G, \mathbb{C}^{s\times r})$ and for all $\mathbf{A}$, $ \mathbf{B} \in M_s(\mathbb{C})$.

\item $\langle \mathbf{f},  \mathbf{g}\rangle =\langle \mathbf{g},  \mathbf{f}\rangle^{*}$ for all  $\mathbf{f},  \mathbf{g} \in L^2(G, \mathbb{C}^{s\times r})$.
\end{enumerate}
Let  tr$A$ denote trace of the matrix $A$.  The scalar-valued  function  $\langle \cdot,\cdot \rangle_{L^2}$ defined by
 \begin{align}\label{indenIIhcI}
  \langle \mathbf{f},\mathbf{g}\rangle_{L^2} = \text{tr}\langle \mathbf{f},\mathbf{g}\rangle, \  \ \mathbf{f},  \  \mathbf{g} \in L^2(G, \mathbb{C}^{s\times r})
   \end{align}
 is an inner product, in the usual sense, on  $L^2(G, \mathbb{C}^{s\times r})$. Further,  $L^2(G, \mathbb{C}^{s\times r})$ becomes a Hilbert space with respect to the inner product given in \eqref{indenIIhcI}. Furthermore, the inner product defined in \eqref{indenIIhcI} induces the  Frobenius norm: $\|\mathbf{f}\|  = \sqrt{\langle \mathbf{f},\mathbf{f} \rangle_{L^2}}, \ \mathbf{f} \in L^2(G, \mathbb{C}^{s\times r})$. Note that $G$ is metrizable and $\sigma$-compact is equivalent to $L^2(G)$ is separable, hence $L^2(G,\mathbb{C}^{s\times r})$ is separable.\\

 \textbf{Adjointable operators on $L^2(G, \mathbb{C}^{s\times r})$}: We say that a bounded linear operator $U$ acting on $L^2(G, \mathbb{C}^{s\times r})$ is  adjointable with respect to the matrix-valued inner product, defined in \eqref{dminp}, if
  $\langle U\textbf{f}, \textbf{g}\rangle=\langle \textbf{f}, U^*\textbf{g}\rangle$ for all \textbf{f}, $\textbf{g} \in L^2(G, \mathbb{C}^{s\times r})$, where  $U^*$ is the Hilbert-adjoint operator of $U$.
  It is noteworthy that a bounded linear operator on  $L^2(G, \mathbb{C}^{s\times r})$ may not be adjointable with respect to the matrix-valued inner product.  This is justified in the following example:

\begin{exa}\label{remb}
Let $G= (\mathbb{R}, +)$ be the additive group of real numbers and $\mu_{G}$ denote the Lebesgue measure on the $\sigma$-algebra of Lebesgue measurable sets of $G$. Consider a bounded, linear and bijective operator $U: L^2(G, \mathbb{C}^{2\times 2}) \rightarrow  L^2(G, \mathbb{C}^{2\times 2})$  given by
\begin{align*}
U: \textbf{f} \mapsto
 \begin{bmatrix}
  f_{12}&f_{11}\\
  f_{21}&f_{22}
  \end{bmatrix}, \ \textbf{f} = \big[f_{ij}\big]_{1 \leq i, j \leq2} \in L^2(G, \mathbb{C}^{2\times 2}).
\end{align*}
Then, $U$ is not adjointable with respect to the matrix-valued inner product on $L^2(G, \mathbb{C}^{2\times 2})$. In fact, for all $ \textbf{f}$,  $\textbf{g} \in L^2(G, \mathbb{C}^{2\times 2})$, we have
\begin{align*}
 \text{tr} \langle U \textbf{f}, \textbf{g}\rangle
= \text{tr}  \int_{G} \begin{bmatrix}
 f_{12}&f_{11}\\
  f_{21}&f_{22}
  \end{bmatrix}\begin{bmatrix}
 \overline{g_{11}} & \overline{g_{21}}\\
  \overline{g_{12}} & \overline{g_{22}}
  \end{bmatrix} d\mu_{G}
=  \text{tr}  \int_{G} \begin{bmatrix}
 f_{11} &f_{12}\\
  f_{21}&f_{22}
  \end{bmatrix}\begin{bmatrix}
\overline{g_{12}}&\overline{g_{21}}\\
\overline{g_{11}}&\overline{g_{22}}
  \end{bmatrix} d\mu_{G}
= \text{tr} \langle \textbf{f}, U \textbf{g}\rangle.
\end{align*}
This implies that $U^*=U$. However, the equality $ \langle U \textbf{f}, \textbf{g}\rangle =  \langle \textbf{f},U^* \textbf{g}\rangle$ does not hold for all  $\textbf{f}$, $\textbf{g} \in L^2(G, \mathbb{C}^{2\times 2})$. Hence, $U$ is not adjointable with respect to the matrix-valued inner product.
\end{exa}
We give the following lemma,  which  is based on the adjointability of bounded linear operators on $L^2(G, \mathbb{C}^{s\times r})$ with respect to the matrix-valued inner product. A simple computation gives its proof.
\begin{lem}\label{lemb}
Let $U,  V \in L^2(G, \mathbb{C}^{s\times r})$ be bounded linear operators that are  adjointable with respect to the matrix-valued inner product. Then
\begin{enumerate}[$(1)$]
\item $U^{-1}$ is adjointable with respect to the  matrix-valued inner product provided $U^{-1}$ exists.
\item $U+V$ is adjointable with respect to the matrix-valued inner product.
\item $UV$ is adjointable with respect to the matrix-valued inner product.
\end{enumerate}
\end{lem}
The following lemma plays a key role in the characterization of positive operators on $L^2(G, \mathbb{C}^{s\times r})$.
\begin{lem}\cite{Jyoti6}\label{lem1}
Let $T$ be a linear operator acting on $L^2(G, \mathbb{C}^{s\times r})$ such that $\langle T \textbf{f}, \textbf{g}\rangle=\langle \textbf{f}, T^*\textbf{g}\rangle$ for all $\textbf{f}, \, \textbf{g} \in L^2(G, \mathbb{C}^{s\times r})$. Then, $T(M\textbf{f})=M T(\textbf{f})$ for all $M \in \mathcal{M}_s(\mathbb{C})$ and for all  $\textbf{f} \in L^2(G, \mathbb{C}^{s\times r})$.
\end{lem}
The  authors have recently investigated  matrix-valued Riesz  bases and frames for  the space  $L^2(G, \mathbb{C}^{s\times r})$ in \cite{JV9}. Necessary and sufficient conditions for the existence of matrix-valued Riesz bases in the space $L^2(G, \mathbb{C}^{s\times r})$ can be found in   \cite{JV9}. Frame conditions in terms of matrix-valued Riesz bases are also discussed in \cite{JV9}.  First, we recall the definition of a matrix-valued orthonormal basis of the space $L^2(G, \mathbb{C}^{s\times r})$.
\begin{defn}\cite{Jyoti6}\label{don}
A \emph{matrix-valued orthonormal basis} for $ L^2(G, \mathbb{C}^{s\times r})$ is a countable family of matrix-valued functions $\{E_k\}_{k \in \mathbb{N}} \subset L^2(G, \mathbb{C}^{s\times r}) $ satisfying the following two properties:
\begin{enumerate}
\item $\langle E_k, E_j \rangle=\begin{cases}
\textbf{I}_{s\times s} , \  k=j\\
\textbf{O}_{s\times s} , \ k \neq j
\end{cases}, \ \text{for all} \  j, \, k \in \mathbb{N}.$
\item Every $\textbf{f} \in L^2(G, \mathbb{C}^{s\times r})$ can be written as $\textbf{f}=\sum\limits_{k \in \mathbb{N}}\langle \textbf{f}, E_k \rangle E_k$, where the series  converges in the  Frobenius norm given in \eqref{eqnorm1}.
\end{enumerate}
\end{defn}
Now, we give the definition of a matrix-valued Riesz basis of the space $L^2(G, \mathbb{C}^{s\times r})$.
\begin{defn}\cite{JV9} \label{def3.2x}
Let $\{E_k\}_{k \in \mathbb{N}}$ be  a matrix-valued orthonormal basis of $L^2(G, \mathbb{C}^{s\times r})$ and let $U$ be a bounded, linear, and bijective operator acting on  $L^2(G, \mathbb{C}^{s\times r})$ that  is adjointable with respect to the matrix-valued inner product, that is,
$\langle U\textbf{f}, \textbf{g}\rangle=\langle \textbf{f}, U^*\textbf{g}\rangle$ for all $\textbf{f}$, $\textbf{g} \in L^2(G, \mathbb{C}^{s\times r})$. A collection of the form $\{U E_k\}_{k \in \mathbb{N}}$  is called a \emph{matrix-valued Riesz basis} for $L^2(G, \mathbb{C}^{s\times r})$.
\end{defn}

\begin{rem}
Compared to the conventional definition of Riesz basis in a separable Hilbert space, Definition \ref{def3.2x} adds the condition $\langle U \textbf{f}, \textbf{g}\rangle=\langle \textbf{f}, U^*\textbf{g}\rangle$ for all $\textbf{f}$, $\textbf{g} \in L^2(G, \mathbb{C}^{s\times r})$. The proof that a matrix-valued Riesz basis for $L^2(G, \mathbb{C}^{s\times r})$ is a frame for $L^2(G, \mathbb{C}^{s\times r})$ depends on this condition, which makes it crucial. For justification, see Example \ref{Hdaki}. We recall that a collection of functions $\{\textbf{F}_k\}_{k \in \mathbb{N}} \subset L^2(G, \mathbb{C}^{s\times r})$ is called a  \emph{matrix-valued frame} for $L^2(G, \mathbb{C}^{s\times r})$ if
\begin{align*}
a_o\|\textbf{f}\|^2 \le \sum_{k\in \mathbb{N}}\|\langle \textbf{f},\textbf{F}_k \rangle\|^2 \le b_o\|\textbf{f}\|^2, \  \textbf{f} \in L^2(G, \mathbb{C}^{s\times r}),
\end{align*}
holds for some  $a_o$, $b_o \in (0, \infty)$. The scalars $a_o$ and $b_o$ are called the \emph{lower frame bound} and \emph{upper frame bound}, respectively, of  $\{\textbf{F}_k\}_{k \in \mathbb{N}}$. The  frame $\{\textbf{F}_k\}_{k \in \mathbb{N}}$ is said to be \emph{tight} if $a_o = b_o$ and \emph{Parseval} if $a_o = b_o =1$.
\end{rem}
The following example shows that unlike the case of standard orthonormal bases and frames the image of a matrix-valued orthonormal basis of  $L^2(G, \mathbb{C}^{s\times r})$ under a bijective, bounded,  and linear operator acting on  $L^2(G, \mathbb{C}^{s\times r})$ may not constitute a frame for  $L^2(G, \mathbb{C}^{s\times r})$.
\begin{exa}\label{Hdaki}
Consider the LCA group $G$ and  bijective, bounded, and  linear operator $U$ given in Example \ref{remb} which is not adjointable with respect to the matrix-valued inner product. Let $\{e_k\}_{k \in \mathbb{N}}$ be a given orthonormal basis  of $L^2(G)$ and $E_k= \text{diag}(e_k, e_k)$, $k \in \mathbb{N}$. Then, $\{E_k\}_{k \in \mathbb{N}}$ is a matrix-valued orthonormal basis for $L^2(G, \mathbb{C}^{2\times 2})$. But, $\{UE_k\}_{k \in \mathbb{N}}$ is not a matrix-valued Riesz basis for $L^2(G, \mathbb{C}^{2\times 2})$. In fact, $\{UE_k\}_{k \in \mathbb{N}}$ is not even a frame for $L^2(G, \mathbb{C}^{2\times 2})$. Indeed, consider a non-zero function $f \in L^2(G) $.
Then,   $ \textbf{f}=\begin{bmatrix}
0 &0\\
f & 0
\end{bmatrix}$ is a non-zero matrix-valued function in $L^2(G,\mathbb{C}^{2\times 2})$ such that
\begin{align*}
\sum_{k \in \mathbb{N}}\Big\|\Big\langle \textbf{f}, U E_k \Big\rangle \Big\|^2
=\sum_{k\in \mathbb{N}}\Big\|\int\limits_{G} \begin{bmatrix}
0 & 0\\
f & 0
\end{bmatrix} \begin{bmatrix}
0 & 0\\
\overline{e}_k & \overline{e}_k
\end{bmatrix} dm_{G} \Big\|^2
=0.
\end{align*}
Thus,  $\{UE_k\}_{k \in \mathbb{N}}$ is not complete in $L^2(G, \mathbb{C}^{2\times 2})$, hence not a frame for $L^2(G, \mathbb{C}^{2\times 2})$.
\end{exa}
Example \ref{Hdaki} also shows that the  image of a matrix-valued orthonormal basis of $L^2(G, \mathbb{C}^{s\times r})$ under a bounded, bijective, and  linear operator acting on $L^2(G, \mathbb{C}^{s\times r})$ may not be a matrix-valued Riesz basis.
For  various types of matrix-valued frames and  bases in matrix-valued function spaces, we refer to  recent papers by the authors \cite{JVK00, JV1, JV9}.

\section{Main Results}\label{secIII}
We begin by describing  classes of operators on $L^2(G, \mathbb{C}^{s\times r})$ that give matrix-valued Riesz bases from orthonormal bases of the space $L^2(G, \mathbb{C}^{s\times r})$. Firstly, we see that if $\{UE_k\}_{k \in \mathbb{N}}$ is a matrix-valued Riesz basis, then $\{U^*UE_k\}_{k \in \mathbb{N}}$ is also a matrix-valued Riesz basis for $L^2(G, \mathbb{C}^{s\times r})$. The matrix-valued adjointability of $U^*U$ follows from Lemma \ref{lemb} and the  bijectivity of $U^*U$ is  obvious. Next, we observe that the sum of two matrix-valued Riesz bases need not be a matrix-valued Riesz basis since the sum of two bijective operators need not be bijective. However, if $\{(I+U^*U)E_k\}_{k \in \mathbb{N}}$ is considered, then this sum becomes a matrix-valued Riesz basis. The matrix-valued adjointability of $I+U^*U$ follows from Lemma \ref{lemb} and the bijectivity of $I+U^*U$ follows from Lemma \ref{lem1a} since
$U^*U$ is a positive operator. In general, we have the following result which is a direct consequence of Lemma \ref{lem1a} and Lemma \ref{lemb}.
\begin{thm}\label{thpi}
Let $T,  S$ be bounded, linear, and  positive operators on $ L^2(G, \mathbb{C}^{s\times r})$ which are adjointable with respect to matrix-valued inner product.  Let $\{E_k\}_{k \in \mathbb{N}}$ be a matrix-valued orthonormal basis for $L^2(G, \mathbb{C}^{s\times r})$. Then,
\begin{enumerate}[$(1)$]
\item $\big\{(I+T)E_k\big\}_{k \in \mathbb{N}}$ and $\big\{(I+S)E_k\big\}_{k \in \mathbb{N}}$ are matrix-valued Riesz bases for $L^2(G, \mathbb{C}^{s\times r})$.
\item $\big\{(I+T+S)E_k\big\}_{k \in \mathbb{N}}$ is a matrix-valued Riesz basis for $L^2(G, \mathbb{C}^{s\times r})$.
\item $\big\{(I+TS)E_k\big\}_{k \in \mathbb{N}}$ is a matrix-valued Riesz basis for $L^2(G, \mathbb{C}^{s\times r})$ if $TS=ST$.
\end{enumerate}
\end{thm}
\begin{cor}\label{cose}
Let $T$ be a bounded, linear, and  positive operator acting  on $ L^2(G, \mathbb{C}^{s\times r})$ which is adjointable with respect to matrix-valued inner product and $\{E_k\}_{k \in \mathbb{N}}$ be a matrix-valued orthonormal basis for $L^2(G, \mathbb{C}^{s\times r})$. Then, $\big\{(I+T+T^2+\cdots +T^n)E_k\big\}_{k \in \mathbb{N}}$ is a matrix-valued Riesz basis for $L^2(G, \mathbb{C}^{s\times r})$ for every $n \in \mathbb{N}$.
\end{cor}
Although $\big\{(I+T+T^2+\cdots +T^n)E_k\big\}_{k \in \mathbb{N}}$ is a matrix-valued Riesz basis for $L^2(G, \mathbb{C}^{s\times r})$ for every $n \in \mathbb{N}$ in Corollary \ref{cose}, the sequence $\big\{(I+T+T^2+\cdots )E_k\big\}_{k \in \mathbb{N}}$ need not be a matrix-valued Riesz basis for $L^2(G, \mathbb{C}^{s\times r})$. In the following result, we will show that $\big\{(I+T+T^2+\cdots )E_k\big\}_{k \in \mathbb{N}}$ becomes a matrix-valued Riesz basis under suitable conditions on the operator $T$. As a side note, the positivity of the operator $T$ is not required to obtain the result.
\begin{thm}\label{thm3.3I}
Let $T$ be a bounded linear operator on  $L^2(G, \mathbb{C}^{s\times r})$ that is adjointable with respect to matrix-valued inner product. If $\|T\| <1$, then $\big\{(I+T+T^2+\cdots )E_k\big\}_{k \in \mathbb{N}}$ is a matrix-valued Riesz basis for $L^2(G, \mathbb{C}^{s\times r})$ where $\{E_k\}_{k \in \mathbb{N}}$ is a matrix-valued orthonormal basis for $L^2(G, \mathbb{C}^{s\times r})$.
\end{thm}
\begin{proof}
Since $\|T\| <1$,  the operator $I-T$ is invertible and $(I-T)^{-1}=\sum_{n=0}^\infty T^n$. This implies that $\sum_{n=0}^\infty T^n$ is a bounded, linear, and  bijective operator on  $L^2(G, \mathbb{C}^{s\times r})$. Also, $(I-T)^{-1}$ is adjointable with respect to matrix-valued inner product by Lemma \ref{lemb}. Hence, $\big\{\big(\sum_{n=0}^\infty T^n\big)E_k\big\}_{k \in \mathbb{N}}$ is a matrix-valued Riesz basis for $L^2(G, \mathbb{C}^{s\times r})$. The proof is complete.
\end{proof}
In order to get another result based on positivity of an operator, we first observe that if $\{TE_k\}_{k \in \mathbb{N}}$ is a Riesz basis for the space $L^2(G)$ for a bounded, bijective, and positive operator $T$ on $L^2(G)$, then $\{T^{1/2}E_k\}_{k \in \mathbb{N}}$ is also a Riesz basis for the space $L^2(G)$ since the operator $T^{1/2}$ is bijective by Theorem \ref{thpo}. This arises a natural question about the extension of the observation made above to matrix-valued Riesz basis in $L^2(G, \mathbb{C}^{s\times r})$. The only part to be checked here is the matrix-valued adjointability of $T^{1/2}$. We see that the matrix-valued adjointability of $T$ need not imply the matrix-valued adjointability of a square root of $T$. Consider the operator $U$ in Example \ref{remb} satisfying $UU=I$. Clearly, the identity operator $I$ is adjointable with respect to matrix-valued inner product but $U$, a square root of $I$, is not adjointable with respect to matrix-valued inner product. In the view of this argument, we give the following result related to the square root of a positive operator on $L^2(G, \mathbb{C}^{s\times r})$. The following lemma is needed in the proof:
\begin{lem}\label{lemcs}
Let $\textbf{f}$, $\textbf{g}\in L^2(G, \mathbb{C}^{s\times r})$. Then, $\| \langle \textbf{f}, \textbf{g}\rangle \| \le r s \| \textbf{f} \| \| \textbf{g} \|$ where $\|.\|$ is the Frobenius norm on the space $\mathcal{M}_s(\mathbb{C})$.
\end{lem}
\begin{proof}
For any $\textbf{f}$,  $\textbf{g} \in L^2(G, \mathbb{C}^{s\times r})$, we compute
\begin{align*}
\|\langle  \textbf{f},\textbf{g} \rangle \|^2 &= \big|\sum\limits_{n=1}^r\int_{G}f_{1n}(x)\overline{g_{1n}(x)} \ d\mu_{G} \big|^2+\cdots+\big|\sum\limits_{n=1}^r \int_{G}f_{sn}(x)\overline{g_{1n}(x)} \ d\mu_{G}\big|^2\\
&\quad +\big|\sum\limits_{n=1}^r \int_{G}f_{1n}(x)\overline{g_{2n}(x)} \ d\mu_{G}\big|^2+\cdots+\big|\sum\limits_{n=1}^r \int_{G}f_{sn}(x)\overline{g_{2n}(x)} \ d\mu_{G}\big|^2+\cdots\\
&\quad  +\big|\sum\limits_{n=1}^r \int_{G}f_{1n}(x)\overline{g_{sn}(x)} \ d\mu_{G}\big|^2+\cdots+\big|\sum\limits_{n=1}^r \int_{G}f_{sn}(x)\overline{g_{sn}(x)} \ d\mu_{G}\big|^2 \\
& \leq  r \Big(\sum\limits_{n=1}^r \big|\int_{G}f_{1n}(x)\overline{g_{1n}(x)} \ d\mu_{G}\big|^2+\cdots+\sum\limits_{n=1}^r \big|\int_{G}f_{sn}(x)\overline{g_{1n}(x)} \ d\mu_{G}\big|^2\\
&\quad  +\sum\limits_{n=1}^r \big| \int_{G}f_{1n}(x)\overline{g_{2n}(x)} \ d\mu_{G}\big|^2+\cdots+\sum\limits_{n=1}^r \big|\int_{G}f_{sn}(x)\overline{g_{2n}(x)} \ d\mu_{G}\big|^2+\cdots\\
&\quad +\sum\limits_{n=1}^r \big|\int_{G}f_{1n}(x)\overline{g_{sn}(x)} \ d\mu_{G}\big|^2+\cdots+\sum\limits_{n=1}^r \big|\int_{G}f_{sn}(x)\overline{g_{sn}(x)} \ d\mu_{G}\big|^2 \Big)\\
& =  r \sum\limits_{i=1}^s \sum\limits_{n=1}^r \Big(\big|\int_{G}f_{in}(x)\overline{g_{1n}(x)} \ d\mu_{G}\big|^2
\quad  + \big| \int_{G}f_{in}(x)\overline{g_{2n}(x)} \ d\mu_{G}\big|^2\\
&\quad  + \cdots+\big|\int_{G}f_{in}(x)\overline{g_{sn}(x)} \ d\mu_{G}\big|^2 \Big)\\
& = r \sum\limits_{i=1}^s \sum\limits_{n=1}^r \sum\limits_{i=1}^s \Big|\int_{G}f_{in}(x)\overline{g_{in}(x)} \ d\mu_{G}\Big|^2 \\
& = r \sum\limits_{i=1}^s \sum\limits_{n=1}^r \sum\limits_{i=1}^s |\langle f_{in}, g_{in} \rangle |^2
 \le r \sum\limits_{i=1}^s \sum\limits_{n=1}^r \sum\limits_{i=1}^s \| f_{in}\|^2 \| g_{in}\|^2
 \le r^2 s^2 \| \textbf{f} \|^2 \| \textbf{g} \|^2.
\end{align*}
This gives the required inequality.
\end{proof}
Next, we give the adjointability  of square root of a positive operator on $L^2(G, \mathbb{C}^{s\times r})$ with respect to the matrix-valued inner product.
\begin{prop}\label{prosq}
Let $T$ be a bounded, linear, and positive operator acting on $L^2(G, \mathbb{C}^{s\times r})$ which satisfies $\langle T\textbf{f}, \textbf{g}\rangle=\langle \textbf{f}, T^*\textbf{g}\rangle$ for all  $\textbf{f}$, $\textbf{g} \in L^2(G, \mathbb{C}^{s\times r})$. Then,
$\big\langle T^{1/2}\textbf{f}, \textbf{g}\big\rangle=\big\langle \textbf{f}, (T^{1/2})^*\textbf{g}\big\rangle$ for all  $\textbf{f}$,   $\textbf{g} \in L^2(G, \mathbb{C}^{s\times r})$.
\end{prop}
\begin{proof}
By Theorem \ref{thpo}, there is a sequence $\{\mathcal{A}_k\}_{k \in \mathbb{N}}$ of polynomials in $T$ which converges to the operator $T^{1/2}$. That is,
$\lim\limits_{k \to \infty}\mathcal{A}_k \textbf{f}=T^{1/2} \textbf{f}$, for every $\textbf{f} \in L^2(G, \mathbb{C}^{s\times r})$. Further, for each $k \in \mathbb{N}$, Lemma \ref{lemb} gives the  adjointability  of  $\mathcal{A}_k$ with respect to the matrix-valued inner product.
 By Lemma \ref{lemcs}, for every $\textbf{f}$, $\textbf{g} \in L^2(G, \mathbb{C}^{s\times r})$,  we have
\begin{align*}
\big\| \big\langle \mathcal{A}_k \textbf{f}, \textbf{g}\big\rangle - \big\langle T^{1/2} \textbf{f}, \textbf{g}\big\rangle\big\|
= \big\| \big\langle \mathcal{A}_k \textbf{f} - T^{1/2} \textbf{f}, \textbf{g}\big\rangle\big\|
 \le rs \big\| \mathcal{A}_k \textbf{f} - T^{1/2} \textbf{f} \big\|\big\| \textbf{g}\big\|
\end{align*}
which implies that
$\lim\limits_{k \to \infty}\big\langle\mathcal{A}_k \textbf{f}, \textbf{g}\big\rangle=\big\langle T^{1/2}\textbf{f}, \textbf{g}\big\rangle$. Using this, we compute
\begin{align*}
\big\langle T^{1/2}\textbf{f}, \textbf{g}\big\rangle
&= \big\langle \lim_{k \to \infty}\mathcal{A}_k \textbf{f}, \textbf{g}\big\rangle\\
&=\lim_{k \to \infty}\big\langle\mathcal{A}_k \textbf{f}, \textbf{g}\big\rangle\\
&=\lim_{k \to \infty}\big\langle\textbf{f}, {\mathcal{A}_k}^* \textbf{g}\big\rangle\\
&=\big\langle\textbf{f}, \lim_{k \to \infty}{\mathcal{A}_k}^* \textbf{g}\big\rangle\\
&=\big\langle\textbf{f}, (T^{1/2})^* \textbf{g}\big\rangle \ \text{for all} \ \textbf{f}, \ \textbf{g} \in L^2(G, \mathbb{C}^{s\times r}).
\end{align*}
Hence, $T^{1/2}$ is adjointable with respect to the  matrix-valued inner product.
\end{proof}
As a consequence of Proposition \ref{prosq}, we get the following interplay between matrix-valued Riesz bases of the space $L^2(G, \mathbb{C}^{s\times r})$ in terms of positive operators  on $L^2(G, \mathbb{C}^{s\times r})$.
\begin{thm}\label{thpsq}
Let $T$ be a bounded, linear, and bijective operator acting on $L^2(G, \mathbb{C}^{s\times r})$  which is adjointable with respect to matrix-valued inner product. If $T$ is positive, then $\{TE_k\}_{k \in \mathbb{N}}$ is a matrix-valued Riesz basis for $L^2(G, \mathbb{C}^{s\times r})$ if and only if $\{T^{1/2}E_k\}_{k \in \mathbb{N}}$ is a matrix-valued Riesz basis for $L^2(G, \mathbb{C}^{s\times r})$.
\end{thm}
\begin{cor}\label{cop}
Let $\{UE_k\}_{k \in \mathbb{N}}$ be a matrix-valued Riesz basis for $L^2(G, \mathbb{C}^{s\times r})$. Then, $\{(U^*U)^{1/2}E_k\}_{k \in \mathbb{N}}$ is a matrix-valued Riesz basis for $L^2(G, \mathbb{C}^{s\times r})$.
\end{cor}
We provide the demonstration of Theorem \ref{thpsq} with the following example. We will use the position operator defined from $L^2(0,1)$ to itself by $Tx(t)=tx(t), \ t \in (0,1)$. This operator plays a significant role in quantum theory.
\begin{exa}\label{expo}
Let  $\Theta : L^2((0,1), \mathbb{C}^{2\times 2}) \to L^2((0,1), \mathbb{C}^{2\times 2})$ be defined as
\begin{align*}
\Theta\textbf{f}(\xi)=\xi \textbf{f}(\xi), \ \xi \in (0,1), \ \textbf{f}\in L^2((0,1), \mathbb{C}^{2\times 2}).
\end{align*}
Then,  $\Theta \begin{bmatrix}
f_{11}(\xi) & f_{12}(\xi)\\
f_{21}(\xi)& f_{22}(\xi)
 \end{bmatrix}=
 \begin{bmatrix}
 \xi f_{11}(\xi)& \xi f_{12}(\xi)\\
 \xi f_{21}(\xi)&\xi f_{22}(\xi)
  \end{bmatrix}, \,  \, \textbf{f} = \big[f_{ij}\big]_{1\leq  i, j \leq2}\in L^2((0,1), \mathbb{C}^{2\times 2})$.
Clearly, $\Theta$ is a linear, bounded, and bijective operator on $L^2((0,1), \mathbb{C}^{2\times 2})$.

 In order to find the Hilbert-adjoint operator of $\Theta$, for every  $\textbf{f}$, $\textbf{g} \in L^2((0,1), \mathbb{C}^{2\times 2})$,  we have
\begin{align*}
 \text{tr} \langle \Theta \textbf{f}, \textbf{g}\rangle &=  \text{tr}  \int_{0}^1 \xi \textbf{f}(\xi)\textbf{g}^*(\xi)d\xi
=\text{tr}  \int_{0}^1  \textbf{f}(\xi)\xi \textbf{g}^*(\xi)d\xi
= \text{tr} \langle\textbf{f},  \Theta \textbf{g}\rangle.
\end{align*}
This implies that $\Theta^*=\Theta$. Moreover, $\langle \Theta \textbf{f}, \textbf{g}\rangle=\langle\textbf{f},  \Theta \textbf{g}\rangle , \ \textbf{f}, \textbf{g} \in L^2((0,1), \mathbb{C}^{2\times 2})$. That is,  $\Theta$ is adjointable with respect to matrix-valued inner product.

 For any $\textbf{f}\in L^2((0,1), \mathbb{C}^{2\times 2})$, we have
\begin{align*}
 \text{tr} \langle \Theta \textbf{f}, \textbf{f}\rangle
&= \text{tr}  \int_{0}^1 \begin{bmatrix}
\xi f_{11}(\xi)&\xi f_{12}(\xi)\\
\xi  f_{21}(\xi)&\xi f_{22}(\xi)
  \end{bmatrix}\begin{bmatrix}
 \overline{f_{11}(\xi)} & \overline{f_{21}(\xi)}\\
  \overline{f_{12}(\xi)} & \overline{f_{22}(\xi)}
  \end{bmatrix} d\xi \\
  &=  \text{tr}  \int_{0}^1 \begin{bmatrix}
 \xi f_{11}(\xi) \overline{f_{11}(\xi)}+ \xi f_{12}(\xi) \overline{f_{12}(\xi)} &   \xi f_{11}(\xi) \overline{f_{21}(\xi)}+ \xi f_{12}(\xi) \overline{f_{22}(\xi)}\\
 \xi f_{21}(\xi) \overline{f_{11}(\xi)}+ \xi f_{22}(\xi) \overline{f_{12}(\xi)} &   \xi f_{21}(\xi) \overline{f_{21}(\xi)}+ \xi f_{22}(\xi) \overline{f_{22}(\xi)}\\
  \end{bmatrix} d\xi\\
&= \int_{0}^1 \xi \big( |f_{11}(\xi)|^2+ |f_{12}(\xi)|^2 + |f_{21}(\xi)|^2+  |f_{22}(\xi)|^2\big) d\xi \\
& \ge 0.
\end{align*}
This implies that $\Theta$ is positive.

 It can be easily seen that the operator $\mathcal{W} : L^2((0,1), \mathbb{C}^{2\times 2}) \to L^2((0,1), \mathbb{C}^{2\times 2})$ defined by
\begin{align*}
\mathcal{W}\textbf{f}(\xi)=\sqrt{\xi} \ \textbf{f}(\xi), \ \xi \in (0,1), \ \textbf{f}\in L^2((0,1), \mathbb{C}^{2\times 2})
\end{align*}
is a linear, bounded, and positive operator such that $\mathcal{W}\mathcal{W}=\Theta$ and hence $\mathcal{W}=\Theta^{1/2}$.

Using the orthonormal basis $\{e^{2\pi i k (\cdot)}\}_{k \in \mathbb{Z}}$ of $L^2(0,1)$ (see details in \cite{Heil}), define
\begin{align*}
E_k=\begin{bmatrix}
 e^{2\pi i k (\cdot)} &0\\
  0&e^{2\pi i k (\cdot)}
  \end{bmatrix}, \ k \in \mathbb{Z}.
\end{align*}
Then, $\{E_k\}_{k \in \mathbb{Z}}$ is a matrix-valued orthonormal basis for $L^2((0,1), \mathbb{C}^{2\times 2})$. By Theorem \ref{thpsq}, both families
\begin{align*}
 \left\{ \begin{bmatrix}
\xi  e^{2\pi i k \xi} &0\\
  0&\xi e^{2\pi i k \xi}
  \end{bmatrix} \right\}_{k \in \mathbb{Z}} \ \text{and} \  \left\{ \begin{bmatrix}
\sqrt{\xi}  e^{2\pi i k \xi} &0\\
  0&\sqrt{\xi} e^{2\pi i k \xi}
  \end{bmatrix} \right\}_{k \in \mathbb{Z}}
\end{align*}
  are Riesz bases for $L^2((0,1), \mathbb{C}^{2\times 2})$.
\end{exa}
Now, we observe that if $\{UE_k\}_{k \in \mathbb{N}}$ is a matrix-valued Riesz basis for $L^2(G, \mathbb{C}^{s\times r})$ and $U=U_1U_2$ for some bounded linear  operators $U_1, U_2$ acting on $L^2(G, \mathbb{C}^{s\times r})$. Then, $\{U_1E_k\}_{k \in \mathbb{N}}$ and  $\{U_2 E_k\}_{k \in \mathbb{N}}$ may not be matrix-valued Riesz bases for $L^2(G, \mathbb{C}^{s\times r})$, by Example \ref{remb}. However, due to polar decomposition of operators (details in \cite{Ru}), we can always obtain  two matrix-valued Riesz bases associated with a given matrix-valued Riesz basis. The construction for the same can be found in the following result.
\begin{thm}\label{thmpol}
Let $\{UE_k\}_{k \in \mathbb{N}}$ be a matrix-valued Riesz basis for $L^2(G, \mathbb{C}^{s\times r})$. Then, there exist bounded linear operators $\mathcal{W}$, $\mathcal{P}$ on $L^2(G, \mathbb{C}^{s\times r})$ such that $\{\mathcal{W}E_k\}_{k \in \mathbb{N}}$ and  $\{\mathcal{P} E_k\}_{k \in \mathbb{N}}$ are matrix-valued Riesz bases for $L^2(G, \mathbb{C}^{s\times r})$.
\end{thm}
\begin{proof}
By polar decomposition, the invertible operator $U$ can be written as a composition of two bounded linear invertible operators on $L^2(G, \mathbb{C}^{s\times r})$. To be more specific, $U=\mathcal{W}\mathcal{P}$ where $\mathcal{P}=(U^*U)^{1/2}$ and $\mathcal{W}=U\mathcal{P}^{-1}$.
By Corollary \ref{cop},  $\{\mathcal{P} E_k\}_{k \in \mathbb{N}}$ is a matrix-valued Riesz basis for $L^2(G, \mathbb{C}^{s\times r})$. By Lemma \ref{lemb}, the operator $\mathcal{W}$ is adjointable with respect to matrix-valued inner product and hence the collection  $\{\mathcal{W} E_k\}_{k \in \mathbb{N}}$ is a matrix-valued Riesz basis for $L^2(G, \mathbb{C}^{s\times r})$. The proof is complete.
\end{proof}
One can easily see that if $\{UE_k\}_{k \in \mathbb{N}}$ is a matrix-valued Riesz basis for $L^2(G, \mathbb{C}^{s\times r})$, then $\{U^n E_k\}_{k \in \mathbb{N}}$ is also a matrix-valued Riesz basis for $L^2(G, \mathbb{C}^{s\times r})$ for every $n \in \mathbb{N}$. That is, for an operator used in the construction of a matrix-valued Riesz basis, the nth power of the operator can be used for constructing different matrix-valued Riesz bases,  for any $n \in \mathbb{N}$. In the following result, we give similar argument about the $\frac{1}{2^n}$th power of the underlying operator. The proof can be obtained by applying Theorem \ref{thpsq} repeatedly.
\begin{thm}\label{thmpwr}
Let $\{TE_k\}_{k \in \mathbb{N}}$ be a matrix-valued Riesz basis for $L^2(G, \mathbb{C}^{s\times r})$. If $T$ is positive, then $\{T^{1/2^n} E_k\}_{k \in \mathbb{N}}$ is also a matrix-valued Riesz basis for $L^2(G, \mathbb{C}^{s\times r})$ for every $n \in \mathbb{N}$, where
\begin{align*}
T^{1/2^n}=\left({\Big(T^{1/2}\Big)^{1/2 \underbrace{\cdots}_{n\text{-times}}}}\right)^{1/2}, \ n \in \mathbb{N}.
\end{align*}
\end{thm}
\begin{rem}
In \cite{JV9}, it is proved that a matrix-valued Riesz basis $\{U E_k\}_{k \in \mathbb{N}}$ for $L^2(G, \mathbb{C}^{s\times r})$ is a  frame for $L^2(G, \mathbb{C}^{s\times r})$ with frame bounds $\|U^{-1}\|^{-2}$ and $\|U\|^2$. Therefore, the matrix-valued Riesz basis $\{T^{1/2^n}E_k\}_{k \in \mathbb{N}}$ in Theorem \ref{thmpwr} is a matrix-valued frame for $L^2(G, \mathbb{C}^{s\times r})$ with frame bounds $\frac{1}{\|(T^{1/2^n})^{-1}\|^2}$ and $\|T^{1/2^n}\|^2$. Using $\|T^{1/2}\|=\|T\|^{1/2}$ by Remark \ref{remsq}, we have
\begin{align*}
&\Big\|\big(T^{1/2}\big)^{1/2}\Big\|=\big\|T^{1/2}\big\|^{1/2}=\|T\|^{1/2^2},\\
\intertext{repeating above steps, we get}
&\Big\|T^{1/2^n}\Big\|=\|T\|^{1/2^n},\ n \in \mathbb{N}.
\end{align*}
Now, $T=T^{1/2}T^{1/2}$ gives $T^{-1}=(T^{1/2})^{-1}(T^{1/2})^{-1}$ which further implies that $(T^{1/2})^{-1}=(T^{-1})^{1/2}$. This gives
\begin{align*}
\Big\|(T^{1/2^n})^{-1}\Big\|=\Big\|(T^{-1})^{1/2^n}\Big\|=\|T^{-1}\|^{1/2^n},\ n \in \mathbb{N}.
\end{align*}
Therefore, the frame bounds of $\{T^{1/2^n} E_k\}_{k \in \mathbb{N}}$ are $\frac{1}{\|T^{-1}\|^{1/2^{n-1}}}$ and $\|T\|^{1/2^{n-1}}$. This leads to an important observation regarding the frame properties of a Riesz basis. To be precise, the obtained frame bounds tend to $1$ as $n \to \infty$ and hence the frame $\{T^{1/2^n} E_k\}_{k \in \mathbb{N}}$ approaches to become a Parseval frame as $n \to \infty$.
\end{rem}

Next, it is quite obvious that if $\{TE_k\}_{k \in \mathbb{N}}$ is a matrix-valued Riesz basis for $L^2(G, \mathbb{C}^{s\times r})$, then $\big\{(\|T\| I) E_k\big\}_{k \in \mathbb{N}}$ is also a matrix-valued Riesz basis for $L^2(G, \mathbb{C}^{s\times r})$. Moreover, if $T$ is a positive operator, we have
\begin{align*}
0 \le  \text{tr}\langle T \textbf{f}, \textbf{f} \rangle  \le \|T \textbf{f}\| \|\textbf{f}\|\le\|T \|\|\textbf{f}\|^2= \|T \| \text{tr}\langle \textbf{f}, \textbf{f} \rangle=\text{tr}\langle \|T\| I\textbf{f}, \textbf{f} \rangle \ \text{for all} \ \textbf{f} \in L^2(G, \mathbb{C}^{s\times r}).
\end{align*}
That is, $\text{tr}\langle T \textbf{f}, \textbf{f} \rangle  \le \text{tr}\langle \|T\|I \textbf{f}, \textbf{f} \rangle \ \text{for all} \ \textbf{f} \in L^2(G, \mathbb{C}^{s\times r})$.
This observation can be generalized using the following result given in \cite{PS}.
\begin{prop}\cite[Proposition 3.7]{PS}\label{2ps}
Let $P_1$ and $P_2$ be positive operators on a Hilbert space $\mathcal{H}$ such that $P_2- P_1 \ge 0$. Then, $P_2$ is an invertible operator if the operator $P_1$ is invertible.
\end{prop}
As a consequence of Proposition \ref{2ps}, we get the following:
\begin{thm}
Let $T$ and $\Omega$ be positive operators on $L^2(G, \mathbb{C}^{s\times r})$ that are adjointable with respect to matrix-valued inner product. Let $\{TE_k\}_{k \in \mathbb{N}}$ be a matrix-valued Riesz basis for $L^2(G, \mathbb{C}^{s\times r})$. Then, $\{\Omega E_k\}_{k \in \mathbb{N}}$ is a matrix-valued Riesz basis for $L^2(G, \mathbb{C}^{s\times r})$ if
\begin{align*}
\text{tr}\langle T \textbf{f}, \textbf{f} \rangle  \le \text{tr}\langle \Omega \textbf{f}, \textbf{f} \rangle, \ \textbf{f} \in L^2(G, \mathbb{C}^{s\times r}).
\end{align*}
\end{thm}
It would be interesting to see if any matrix-valued Riesz bases can be obtained without imposing the condition of positivity on operator $T$. We show that a bigger class of operators, that is, the class of self-adjoint operators can be used to construct different matrix-valued Riesz bases for $L^2(G, \mathbb{C}^{s\times r})$. It is quite evident that if $T$  is a self adjoint operator on $L^2(G, \mathbb{C}^{s\times r})$, then the operator $T^2$ is positive and therefore $\big\{(I+T^2)E_k\big\}_{k \in \mathbb{N}}$ is a matrix-valued Riesz basis for $L^2(G, \mathbb{C}^{s\times r})$ for any orthonormal basis $\{E_k\}_{k \in \mathbb{N}}$ of $L^2(G, \mathbb{C}^{s\times r})$ if $T$ is adjointable with respect to matrix-valued inner product, by Lemma \ref{lemb} and Theorem \ref{thpi}. Not only this, two more Riesz bases can be obtained using the self-adjoint operator $T$. Before coming to the main result, we recall  Proposition 3.2 from \cite{PS} which states that every self-adjoint operator can be written as a difference of two positive operators.
\begin{prop} \cite[Proposition 3.2]{PS}\label{propp}
Let $T$ be a bounded, linear, and self-adjoint operator on a Hilbert space $\mathcal{H}$. Then, there exist unique positive operators $P_1$, $P_2$ such that $T=P_1-P_2$ and $P_1 P_2=0$. Precisely, the operators $P_1$, $P_2$ are given by
\begin{align*}
P_1=\frac{1}{2}\big[(T^2)^{1/2}+ T \big], \ P_2=\frac{1}{2}\big[(T^2)^{1/2}-T \big].
\end{align*}
\end{prop}
As a consequence of Proposition \ref{propp}, we get the following result. The proof is based on Proposition \ref{prosq}, Lemma \ref{lemb} and Theorem \ref{thpi}.
\begin{thm}\label{thmmisc24}
Let $T$ be a bounded, linear, and  self-adjoint operator acting on $L^2(G, \mathbb{C}^{s\times r})$ that is adjointable with respect to matrix-valued inner product, and $\{E_k\}_{k \in \mathbb{N}}$ be a matrix-valued orthonormal basis for $L^2(G, \mathbb{C}^{s\times r})$. Then,
\begin{align*}
\Big\{\big(I+\frac{1}{2}\big((T^2)^{1/2}+T \big)\big)E_k\Big\}_{k \in \mathbb{N}} \ \text{and} \ \Big\{\big(I+\frac{1}{2}\big((T^2)^{1/2}-T \big)\big)E_k\Big\}_{k \in \mathbb{N}}
\end{align*}
are matrix-valued Riesz bases for $L^2(G, \mathbb{C}^{s\times r})$.
\end{thm}
Not only this, any linear, bounded, and  self-adjoint operator on $\mathcal{H}$  with norm equal to one can be written as a linear combination of two unitary operators, see Proposition 3.8 in \cite{PS} for details. To be precise, a bounded, linear, and  self-adjoint operator $T$ on $L^2(G, \mathbb{C}^{s\times r})$ with norm one can be written as
\begin{align}\label{eqsa}
T=\frac{1}{2}\Big(T+i\big(I-T^2\big)^{1/2}\Big)+\frac{1}{2}\Big(T-i\big(I-T^2\big)^{1/2}\Big)
\end{align}
where the operators $T \pm i \big(I-T^2\big)^{1/2}$ are unitary and hence bijective. Therefore, we get the following result for constructing matrix-valued Riesz bases for $L^2(G, \mathbb{C}^{s\times r})$. The proof follows by equation \eqref{eqsa}, Lemma \ref{lemb} and Proposition \ref{prosq}.

\begin{thm}\label{thsa2}
Let $T$ be a bounded, linear, and  self-adjoint operator acting on $L^2(G, \mathbb{C}^{s\times r})$ that is adjointable with respect to matrix-valued inner product, and $\{E_k\}_{k \in \mathbb{N}}$ be a matrix-valued orthonormal basis for $L^2(G, \mathbb{C}^{s\times r})$. Then,
\begin{align*}
\Big\{\Big(\frac{1}{\|T\|}\big[T+ i \big(\|T\|^2 I-T^2\big)^{1/2}\big]\Big)E_k\Big\}_{k \in \mathbb{N}} \ \text{and} \  \Big\{\Big(\frac{1}{\|T\|}\big[T- i \big(\|T\|^2 I -T^2\big)^{1/2}\big]\Big)E_k\Big\}_{k \in \mathbb{N}}
\end{align*}
are matrix-valued Riesz bases for $L^2(G, \mathbb{C}^{s\times r})$.
\end{thm}
We now see that even the self-adjointness of a bounded linear operator on  $T$ on $L^2(G, \mathbb{C}^{s\times r})$ can be omitted and we still obtain some matrix-valued Riesz bases for $L^2(G, \mathbb{C}^{s\times r})$. Using the Proposition 3.8 in \cite{PS} again, we see that any bounded linear operator $T$ on $L^2(G, \mathbb{C}^{s\times r})$ can be written as a linear combination of two self-adjoint operators. That is,
\begin{align}\label{eqsa2}
T=\frac{1}{2}(T+T^*)+\frac{i}{2}\Big(\frac{T-T^*}{i}\Big)
\end{align}
where the operators $T+T^*$ and $\frac{T-T^*}{i}=-i(T-T^*)$ are self-adjoint. Therefore, using equation \eqref{eqsa2} and Theorem \ref{thsa2}, we get the following result.
\begin{thm}\label{clarefii}
Let $\{E_k\}_{k \in \mathbb{N}}$ be a matrix-valued orthonormal basis for $L^2(G, \mathbb{C}^{s\times r})$. Let $T$ be a bounded linear operator on $L^2(G, \mathbb{C}^{s\times r})$ that is adjointable with respect to matrix-valued inner product. Then, for each $1 \le i \le 4$ , $\{\Omega_i E_k\Big\}_{k \in \mathbb{N}}$ is a matrix-valued Riesz basis for $L^2(G, \mathbb{C}^{s\times r})$ where $\Omega_i$ are the following operators:
\begin{enumerate}[$(1)$]
 \item $\Omega_1=\frac{1}{\|T+T^*\|}\Big(T+T^*+ i \big(\|T+T^*\|^2I-(T+T^*)^2\big)^{1/2}\Big)$.\\
\item $\Omega_2=\frac{1}{\|T+T^*\|}\Big(T+T^*- i \big(\|T+T^*\|^2I-(T+T^*)^2\big)^{1/2}\Big)$.\\
 \item $\Omega_3=\frac{-i}{\|T-T^*\|}\Big(T-T^*-\big(\|T-T^*\|^2I+(T-T^*)^2\big)^{1/2}\Big)$.\\
\item $\Omega_4=\frac{-i}{\|T-T^*\|}\Big(T-T^*+\big(\|T-T^*\|^2I+(T-T^*)^2\big)^{1/2}\Big)$.\\
\end{enumerate}
\end{thm}

Next, we give definition of the dual Riesz basis of given matrix-valued Riesz basis of $L^2(G, \mathbb{C}^{s\times r})$. Let  $\{\textbf{f}_k\}_{k \in \mathbb{N}}=\{U E_k\}_{k \in \mathbb{N}}$ be a matrix-valued Riesz basis for $L^2(G, \mathbb{C}^{s\times r})$, then
\begin{enumerate}[$(i)$]
  \item
$\big\langle (U^{-1})^*\textbf{f}, \textbf{g} \big\rangle=\big\langle (U^{-1})^*\textbf{f}, U U^{-1}\textbf{g} \big\rangle
=\big\langle (U^{-1}U)^*\textbf{f}, U^{-1}\textbf{g}\big \rangle
=\langle \textbf{f}, U^{-1}\textbf{g} \rangle$ for all $\textbf{f}$,  $\textbf{g} \in L^2(G, \mathbb{C}^{s\times r})$.
Thus,  $(U^{-1})^*$ is a bounded and bijective operator which is adjointable with respect to the matrix-valued inner product.
 \item
$\textbf{f}= U U^{-1}\textbf{f}
=\sum\limits_{k \in \mathbb{N}}\langle U^{-1}\textbf{f}, E_k \rangle \mathcal{U} E_k
 =\sum\limits_{k \in \mathbb{N}}\langle \textbf{f}, (U^{-1})^*E_k \rangle U E_k$ for all  $\textbf{f} \in L^2(G, \mathbb{C}^{s\times r})$.
\end{enumerate}
Choose $\textbf{g}_k= (U^{-1})^*E_k$, $k \in \mathbb{N}$. Then, $\textbf{f}=\sum_{k \in \mathbb{N}}\langle \textbf{f}, \textbf{g}_k \rangle \textbf{f}_k \ \text{for all}\ \textbf{f} \in L^2(G, \mathbb{C}^{s\times r})$.
The sequence $\{\textbf{g}_k\}_{k \in \mathbb{N}}$  is known as  \emph{dual Riesz basis} of $\{\textbf{f}_k\}_{k \in \mathbb{N}}$. We also say that  $\big(\{\textbf{f}_k\}_{k \in \mathbb{N}}, \{\textbf{g}_k\}_{k \in \mathbb{N}}\big)$ is a dual Riesz basis pair.

For a dual Riesz basis pair $\big(\{\textbf{f}_k\}_{k \in \mathbb{N}}, \{\textbf{g}_k\}_{k \in \mathbb{N}}\big)$ in the Hilbert space $L^2(G)$, by Theorem 3.1 in \cite{HolubII}, the operator mapping $\textbf{f}_k$ to $\textbf{g}_k$, $k \in \mathbb{N}$  is a positive operator. This may not be true in the matrix-valued space $L^2(G, \mathbb{C}^{s\times r})$. This is justified in the following example:
\begin{exa}\label{exH}
For an orthonormal basis $\{e_k\}_{k \in \mathbb{N}}$ of $L^2(G)$, let $\textbf{f}_k = \text{diag}(e_k, e_k)$, $k \in \mathbb{N}$.
Then, $\{\textbf{f}_k\}_{k \in \mathbb{N}}$ is a matrix-valued orthonormal basis and hence a matrix-valued Riesz basis for $L^2(G, \mathbb{C}^{2\times 2})$ with the dual Riesz basis $\{\textbf{g}_k\}_{k \in \mathbb{N}}=\{\textbf{f}_k\}_{k \in \mathbb{N}}$.
Define $U: L^2(G, \mathbb{C}^{2\times 2}) \rightarrow  L^2(G, \mathbb{C}^{2\times 2})$ by
\begin{align*}
U: \textbf{f} \mapsto
 \begin{bmatrix}
  f_{11}&f_{21}\\
  f_{12}&f_{22}
  \end{bmatrix}, \ \textbf{f} = \big[f_{ij}\big]_{1 \leq i, j \leq2} \in L^2(G, \mathbb{C}^{2\times 2}).
\end{align*}
Then, $U\textbf{f}_k=\textbf{g}_k, k \in \mathbb{N}$, but $U$ is not a positive operator. Indeed, for an arbitrary but fixed $0 \ne f \in L^2(G)$,  $\textbf{f}=\begin{bmatrix}
 0&f\\
- f&f
\end{bmatrix} \in L^2(G, \mathbb{C}^{2\times 2})$ such that
\begin{align*}
\text{tr} \langle U \textbf{f}, \textbf{f}\rangle
&= \text{tr}  \int\limits_{G} \begin{bmatrix}
  0&-f\\
f&f
  \end{bmatrix}\begin{bmatrix}
 \overline{0} & -\overline{f}\\
  \overline{f} & \overline{f}
  \end{bmatrix} d\mu_{G} = -\|f\|^2 <0.
\end{align*}
Hence, $U$ is not positive.
\end{exa}
Note that the operator $U$ in Example \ref{exH} is not adjointable with respect to the matrix-valued inner product on $L^2(G, \mathbb{C}^{2\times 2})$. We see that Theorem 3.1 of \cite{HolubII} due to Holub holds in the space $L^2(G, \mathbb{C}^{s\times r})$ only if an additional condition of matrix-valued adjointability is imposed on the underlying operator. Precisely, we will prove that the matrix-valued adjointable operator mapping $\textbf{f}_k$ to $\textbf{g}_k$, $k \in \mathbb{N}$  is a positive operator. Conversely, every positive operator, that is bijective and adjointable with respect to matrix-valued inner product, maps some matrix-valued Riesz basis to its matrix-valued dual Riesz basis.

\begin{thm}\label{mainthmI}
Let $T$ be a bounded, linear, and bijective operator acting on $L^2(G, \mathbb{C}^{s\times r})$  which is adjointable with respect to the matrix-valued inner product. Then, $T$ is positive if and only if $T$ maps a matrix-valued Riesz basis of $L^2(G, \mathbb{C}^{s\times r})$ to its matrix-valued dual Riesz basis.
\end{thm}
\begin{proof}
First, let $\{\textbf{f}_k\}_{k \in \mathbb{N}}=\{U E_k\}_{k \in \mathbb{N}}$ and $\{\textbf{g}_k\}_{k \in \mathbb{N}}=\{(U^{-1})^*E_k\}_{k \in \mathbb{N}}$ be a pair of matrix-valued dual Riesz bases and
$T\textbf{f}_k= \textbf{g}_k$, for all $k \in \mathbb{N}$. That is, $TU E_k= (U^{-1})^*E_k, k \in \mathbb{N}$. Using Lemma \ref{lemb} and Lemma \ref{lem1}, for every $\textbf{f} \in L^2(G, \mathbb{C}^{s\times r})$, we compute
\begin{align*}
TU \textbf{f}&= TU \sum\limits_{k \in \mathbb{N}}\langle \textbf{f}, E_k \rangle E_k\\
&=\sum\limits_{k \in \mathbb{N}}\langle \textbf{f}, E_k \rangle TU E_k\\
&=\sum\limits_{k \in \mathbb{N}}\langle \textbf{f}, E_k \rangle (U^{-1})^* E_k\\
&=(U^{-1})^*\sum\limits_{k \in \mathbb{N}}\langle \textbf{f}, E_k \rangle E_k\\
&=(U^{-1})^* \textbf{f}.
\end{align*}
This gives $TU=(U^{-1})^*$, and hence $T=(U^{-1})^* U^{-1}$ which clearly is a positive operator.

Conversely, assume now that $T$ is positive. Then, by Theorem \ref{thpo} and Proposition \ref{prosq}, $\{T^{1/2} E_k\}_{k \in \mathbb{N}}$ is a matrix-valued Riesz basis of $L^2(G, \mathbb{C}^{s\times r})$, where $\{E_k\}_{k \in \mathbb{N}}$ is a matrix-valued orthonormal basis of $L^2(G, \mathbb{C}^{s\times r})$.  The matrix-valued dual Riesz basis of  $\{\textbf{h}_k\}_{k \in \mathbb{N}}=\{T^{1/2} E_k\}_{k \in \mathbb{N}}$ is  given by
\begin{equation}
\{\textbf{q}_k\}_{k \in \mathbb{N}}=\{((T^{1/2})^{-1})^* E_k\}_{k \in \mathbb{N}}=\{(T^{1/2})^{-1} E_k\}_{k \in \mathbb{N}}. \nonumber
\end{equation}
For every $k \in \mathbb{N}$, we have
\begin{align*}
T\textbf{q}_k= T(T^{1/2})^{-1} E_k= T^{1/2}T^{1/2}(T^{1/2})^{-1} E_k=T^{1/2} E_k= \textbf{h}_k.
\end{align*}
Hence, $T$ maps the matrix-valued Riesz basis $\{\textbf{q}_k\}_{k \in \mathbb{N}}$ to its dual Riesz basis $\{\textbf{h}_k\}_{k \in \mathbb{N}}$. The proof is complete.
\end{proof}

\section{Conclusion}
The construction of frames and bases by the action of operators on a given frame, basis, or vector in the given signal space plays  an important role in a variety of problems in frame decompositions of signals \cite{AA3}, new classes of frames from given frames \cite{AA1}, dynamical sampling \cite{AA2, AA4},  etc.  Additionally, the weaving properties of Riesz bases, which are related to distributive signal processing \cite{Deep1, Deep2}, yield deeper insight into the stable analysis of functions in $L^2(G, \mathbb{C}^{s\times r})$ in terms of Riesz bases. In this study, we provided new classes of bounded linear operators on  $L^2(G, \mathbb{C}^{s\times r})$ such that their images on matrix-valued orthonormal bases for $L^2(G, \mathbb{C}^{s\times r})$ constitute bases and frames for the space $L^2(G, \mathbb{C}^{s\times r})$. These classes of operators will be useful in the study of various types of multiresolution analyses \cite{Jind, NPS} and  data analysis in higher-dimensional signal spaces \cite{Zpelle}. A characterization of positive isomorphisms on  $\mathcal{H}$ in terms of Riesz bases of $\mathcal{H}$ due to Holub  \cite{HolubII} holds only for the class of bounded, linear, and bijective operators acting on  $L^2(G, \mathbb{C}^{s\times r})$ that are adjointable with respect to matrix-valued inner product.  Our results will stimulate the development of Riesz bases in the matrix-valued function space $L^2(G, \mathbb{C}^{s\times r})$. For example,  the study of the  near-Riesz basis \cite{Holub} and its relation to frames in  $L^2(G, \mathbb{C}^{s\times r})$ is one of the important directions for research.

\bibliographystyle{amsplain}

\mbox{}

\end{document}